\documentclass[12pt]{article}

      \usepackage{latexsym}
         \usepackage[reqno, namelimits, sumlimits]{amsmath}
         \usepackage{amssymb, amsfonts}
         \usepackage{amsthm}

 \newtheorem{theorem}{Theorem}[section]
 
 \newtheorem{definition}[theorem]{Definition}
 
 \newtheorem{lemma}[theorem]{Lemma}
 
 \newtheorem{remark}[theorem]{Remark}

 \newtheorem{pro}[theorem]{Proposition}

\title{$LlogL$-integrability of the velocity gradient for Stokes system with drifts in $L_\infty (BMO^{-1})$}

 \author{J.~Burczak\footnote{Institute of Mathematics, Polish Academy of Sciences,   \'Sniadeckich 8, 00-656 Warsaw, Poland and Mathematical Institute, University of Oxford, UK, emails: jb@impan.pl and  burczak@maths.ox.ac.uk }, \; G.~Seregin\footnote{Oxford University, UK and St Petersburg Department of Steklov Mathematical Institute, RAS, Russia, email: seregin@maths.ox.ac.uk}
}

\date{}
\begin{document}
\maketitle
\def\Xint#1{\mathchoice
 {\XXint\displaystyle\textstyle{#1}}%
 {\XXint\textstyle\scriptstyle{#1}}%
{\XXint\scriptstyle\scriptscriptstyle{#1}}%
 {\XXint\scriptscriptstyle\scriptscriptstyle{#1}}%
\!\int}
\def\XXint#1#2#3{{\setbox0=\hbox{$#1{#2#3}{\int}$}
 \vcenter{\hbox{$#2#3$}}\kern-.5\wd0}}
 \def\ddashint{\Xint=}
 \def\dashint{\Xint-}

\begin{abstract}
\noindent
For any weak solution of the Stokes system with drifts in $L_\infty (BMO^{-1})$, we prove a reverse H\"older inequality and  $LlogL$-higher integrability of the velocity gradients. 
\end{abstract}

\setcounter{equation}{0}
\section{Introduction}

Let us consider the following $3$D Stokes system with drift
\begin{equation}\label{os1}
\begin{aligned}
\partial_t v + b \cdot \nabla v - \Delta v + \nabla q &= 0, \\
  {\rm div}\, v &= 0,
\end{aligned}
\end{equation}
where $b$ is a given vector field and $v$ and $q$ are unknown velocity field and pressure.

Our interest in \eqref{os1} is related to possible regularity improvements in the Navier-Stokes borderline case $b \in L_\infty ( BMO^{-1})$, at least in the size of a possible singular set. Hence we assume throughout this note that 
\begin{equation}\label{eq:sol}
{\rm div}\, b = 0.
\end{equation}
There are different definitions of the space $BMO^{-1}$, see for example Koch \& Tataru \cite{KochTat01}. In our $3$D case, it is convenient to use the following one: there exists a tensor $d\in BMO$ such that \begin{equation}
	\label{bmodef}
b={\rm div}\,d\end{equation}
 in the sense of distributions, while condition (\ref{eq:sol}) implies its skew-symmetry. Equivalently, there exists a divergence free field $\omega\in BMO$ such that $b={\rm rot}\,\omega$. Then $d_{ij}=\epsilon_{ijk}\omega_k $, where $(\epsilon_{ijk})$ is the Levi-Civita tensor.

 The relationship between $b$ and $d$ shows that  one may recast \eqref{os1} as a generalised Stokes system with the main part $A = Id + D$, where $D=(D_{ijkl})$ with $D_{ijkl}=\delta_{ik}d_{jl}\in L_\infty(BMO)$. A general $A \in L_\infty(BMO)$  is naturally too rough even to define a standard weak solution. But here skew-symmetry comes again to our aid. Namely, we have the following estimate  
 \begin{equation}\label{MazVer}
 	 \int\limits_{\mathbb R^n} (D \nabla u): \nabla v \, dx\leq c\|d\|_{BMO}\|\nabla u\|_2\|\nabla v\|_2
 	 \end{equation} 
 for any $u,v\in C^\infty_{0}(\mathbb R^3)$, which can be deduced from the results of
Maz'ya \& Verbitsky  \cite{MazVer06}. A related discussion may be found in Silvestre, \v{S}ver\'ak, Zlato\v{s} \& coauthor
  \cite{SSSZ}. We give a straightforward proof of (\ref{MazVer}) in the Appendix I for completeness.

It is important to keep in mind that over the entirety of this note, while we refer to $b \in L_\infty ( BMO^{-1})$ satisfying (\ref{eq:sol}), we automatically consider (\ref{bmodef}) with the related $D$.

Among other interesting cases, in which the system  (\ref{os1}) plays an important  part, there is the question about potential Type I blowup of solutions to the Navier-Stokes system, compare the recent paper \cite{SchSer2017} by Schonbek  \& coauthor about a Liouville-type theorem via duality.

For the account of the achievable regularity results for the scalar version of the problem \eqref{os1} with the structural restriction (\ref{eq:sol})  but with no pressure, i.e.
$$
\partial_t u + b \cdot \nabla u - \Delta u= 0, \qquad {\rm div}\, b = 0,
$$
we refer to \cite{SSSZ}. The essence of its results reads: among $L_\infty (X)$ spaces for $b$, $X=BMO^{-1}$ is the widest one, where local `deep' regularity results for $u$ are available (e.g. Harnack inequality) and the choice of $BMO^{-1}$ is close to being sharp. See also Nazarov \& Ural'tseva 
\cite{NazUra11} for $b$ in space-time Morrey spaces on the same scale and Liskevich \& Zhang 
\cite{LisZha04} for similar results under a `form boundedness assumption' on $b$. One should in addition mention Friedlander \& Vicol 
\cite{FriVic11}, where H\"older continuity of solutions to the related Cauchy problem was proved, with $b \in L_\infty (BMO^{-1})$.

In relation to the full system (\ref{os1}-\ref{eq:sol}), the current best result for the associated Cauchy problem is Silvestre \& Vicol 
\cite{SilVic12}. The authors show for $b \in L_p (M^\beta)$, a Lebesgue-Morrey scale of spaces, that there exists a $C (C^\alpha)$ solution. However, for the endpoint of this scale i.e. for $L_\infty (M^{-1})$, $M^{-1} \supset L^3$,  in order to conclude with the same result, an additional smallness assumption is needed (which is automatically satisfied for $C(L^3)$, but not for $L_\infty(L^3)$). For  the local setting, we refer to Zhang 
\cite{Zha04}, where $b$ must belong to a certain Kato class.

Let us conclude with two remarks. Firstly, as already seen above, for a scale of spaces, the regularity results in the endpoint case  $L_\infty (X)$ are substantially more difficult and even likely not always to hold. Secondly, the result of Escauriaza, \v{S}ver\'ak \& coauthor  \cite{EscSerSve03}, where $ b = v \in L_\infty(L_3)$ suffices to obtain regularity, utilises essentially the nonlinear structure. Hence to study regularity of solutions to (\ref{os1}) with   \eqref{eq:sol}, even with $L_\infty (L_3)$, one needs different ideas.

\setcounter{equation}{0}
\section{Main Results}

We write $B(x_0,R)$ for the ball with radius $R$ centred at $x_0 \in \mathbb R^3$. $Q(z_0,R) = B (x_0,R) \times (t_0 - R^2, t_0)$ is the (parabolic) cylinder with its centre $z_0 = (x_0, t_0)$, where $t_0 \in \mathbb R$. 
For an open set $ \Omega \subset \mathbb R^3$ and an interval $]T_1, T_2[$, we write $Q_{T_1,T_2} =  \Omega \times ]T_1, T_2[
$. 

We will use standard function spaces: $L_\infty(]T_1,T_2[\,; L_2 (\Omega)) = L_{2, \infty} (Q_{T_1,T_2})$, $ W_2^{1,0} (Q_{T_1,T_2})=\{v,\nabla v\in L_2(Q_{T_1,T_2})\}$, etc. 

In what follows we always adopt the following convention
\begin{equation}\label{eq:gam}
\Gamma (z,\rho) = \| b \|_{L_\infty ({t - \rho^2, t};  BMO^{-1} (B (x,\rho) ))} = \| d \|_{L_\infty ({t -  \rho^2, t};  BMO (B (x,\rho ))},
\end{equation}
where $d$ is related with $b$ via (\ref{bmodef}). Naturally, the right-hand side of (\ref{eq:gam}) is merely a seminorm for $d$, but the right-hand side is a proper norm for $b$, see e.g. \cite{KochTat01}.

Where there is no danger of confusion, we may sometimes suppress certain indices.

\begin{definition}[Weak solution] \label{sws} Let us fix a space-time domain $Q_{T_1,T_2}$.
A pair $v=(v_i)$ and $q$ is a weak solution to (\ref{os1}) on $Q_{T_1,T_2}$ if and only if
\begin{itemize}
\item[(i)] $v \in L_{2, \infty} (Q_{T_1,T_2}) \cap  W_2^{1,0} (Q_{T_1,T_2})$ \quad and \quad $q \in L_2  (Q_{T_1,T_2})$;
\item[(ii)] $v$ and $q$ satisfy  (\ref{os1}) in the sense of distributions on $Q_{T_1,T_2}$.
\end{itemize}
\end{definition}

\begin{remark}\label{1strem}
The regularity classes appearing in Definition \ref{sws}, in particular $L_2$ for the pressure $q$, agree with the existence result for the Cauchy problem for ({\ref{os1}}) with a solenoidal drift $b \in L_\infty ( BMO^{-1})$, see Appendix II.
\end{remark}

\begin{remark}\label{ee}
Any weak solution to (\ref{os1}-\ref{eq:sol}) on $Q_{T_1,T_2}$ satisfies the following local energy identity
	$$\int\limits_{\Omega}\varphi|v(x,t)|^2dx+2\int\limits^t_0\int\limits_{\Omega}\varphi|\nabla v|^2dxdt'= $$$$= \int\limits^t_0\int\limits_{\Omega}
	(|v|^2(\partial_t+\Delta) \varphi-2D\nabla v:v\otimes \nabla \varphi+2qv\cdot\nabla\varphi)dxdt' $$
	for any $t \in \, ]T_1, T_2[$ and any non-negative $\varphi\in C^\infty_0(Q_{T_1, T_2 +1})$.
\end{remark}
The above remark  follows from (\ref{MazVer}) and standard duality arguments. Observe that it renders a notion of a suitable weak solution redundant in our setting.

Our first result is as follows. 
\begin{pro}\label{rh}
For any $l \in\, ]6/5, 2[$, any weak solution $v$ and $q$ to (\ref{os1}-\ref{eq:sol}) on $Q_{T_1,T_2}$ satisfies 
$$\frac 1{|Q(\rho)|} \int\limits_{Q (z_0,\rho)} |\nabla v|^2  dz \leq $$
\begin{equation}\label{rhin}
\leq C(l) ( \Gamma^5 (z_0,2\rho) +1) \bigg( \frac 1{|Q(2\rho)|} \int\limits_{Q (z_0,2\rho)} |\nabla v |^l dz \bigg)^\frac{2}{l} +\end{equation}
$$+ C \bigg(\frac 1{|Q(2\rho)|} \int\limits_{Q(z_0,2\rho)} |q |  \, dz\bigg)^2$$
on any $Q(z_0,2\rho) \subset Q_{T_1,T_2}$,
with constants $C(l)$ and $C$.
\end{pro}

A simple consequence of Proposition \ref{rh} is as follows.

\begin{remark}\label{liu}
Let $b \in L_\infty (\mathbb R;  BMO^{-1} (\mathbb R^3))$ satisfy (\ref{eq:sol}). Then any weak solution to (\ref{os1}) on $\mathbb R^3 \times \mathbb R$ vanishes.
\end{remark}
Indeed, let  $\Gamma_\infty = \| b \|_{L_\infty (\mathbb R;  BMO^{-1} (\mathbb R^3 ))}$,  $h= |\nabla v |^s$ and $M$ denote the (centred) maximal function with respect to parabolic cylinders (they satisfy the `doubling' assumptions on families of open sets, needed to provide the usual maximal function theory, compare Stein \cite{Stein93},  \S I.1).  Proposition \ref{rh} gives
$$
M (h^\frac{2}{s}) (z) \leq C(s, \Gamma_\infty )M^\frac{2}{s}   (h) (z) +C M^2 (q ) (z) .
$$
The strong $L_p$ estimates for $M$ imply
$$
\int\limits_{\mathbb R^4} M (h^\frac{2}{s})\,  dz  \leq C(s, \Gamma_\infty ) \int\limits_{\mathbb R^4} h^\frac{2}{s}  dz  +C \int\limits_{\mathbb R^4} |q|^2 dz=$$$$ = C(s, \Gamma_\infty ) \int\limits_{\mathbb R^4} |\nabla v |^2 dz+C  \int\limits_{\mathbb R^4}  |q |^2 dz  \leq C.
$$
This means that both $M (h^\frac{2}{s})$ and $h^\frac{2}{p}$ are integrable. On the full space it yields that $h^\frac{2}{s} \equiv 0$, compare \cite{Stein93}, \S I.8.14. Therefore $v$ can only be time-dependant, but then our assumption $v \in L_{2, \infty}$ implies $v \equiv 0$.

Our main result reads
\begin{theorem}\label{log}
Let $b$ satisfy (\ref{eq:sol}). Then, there exists a number $C$, such that any weak solution $v$ and $q$ to (\ref{os1}) in $Q_{T_1,T_2}$ satisfies
$$
 \int\limits_{Q(z_0,r)} |\nabla v|^2 \log \bigg( 1+ \frac{|\nabla v|^2}{ (|\nabla v|^2)_{z_0,r}} \bigg) dz \leq$$$$ \leq C  (1+ \Gamma^5 (z_0,5r) ) \int\limits_{Q(z_0,5r)} |\nabla v|^2 dz + C  \int\limits_{Q(z_0,5r)}   |q |^2 \, dz
$$
for any $Q (z_0, 5r) \Subset Q_{T_1,T_2}$.

Here, $(f)_{z_0,r}$ is the mean value of function $f$ over the parabolic cylinder $Q(z_0,r)$.
\end{theorem}

We would like to notice that, in \cite{ChoeYang}, the authors claim even a stronger result about higher integrability of the velocity gradient.

\setcounter{equation}{0}
\section{Proof of Proposition \ref{rh}}
Over this proof, we will refer at certain times to  \cite{Ser08}. Let us thence initially observe, that however it deals with the case $b=v$, all the computations are in fact performed there for (\ref{os1} - \ref{eq:sol}).

For an $x_0 \in \mathbb R^3$ and $r<R$, let $\varphi_{x_0, r, R} (x)$ be a radial nonnegative smooth  space cut-off function, such that
$$
\varphi_{x_0, r, R} \equiv 1 \;\text{ on }\; B (x_0,r), \qquad  \varphi_{x_0, r, R} \equiv 0 \; \text{ outside }\; B(x_0,R),  $$$$  |\nabla^{i} \varphi_{x_0, r, R} | \leq \frac{C_i}{(R-r)^i}.
$$
Let us introduce the related mean value of a function $f$
$$
f_{x_0, r, R} (t) = \int\limits_{ B(x_0, R)} f (x,t) \varphi^2_{x_0, r, R} (x) \, dx \, \bigg(\,  \int\limits_{ B(x_0,R)}  \varphi^2_{x_0, r, R} (x) \,dx\bigg)^{-1}.
$$
We will also need  a smooth nonnegative time cut-off function $\chi_{t_0, r, R} (t)$ with the following properties
$$
\chi_{t_0, r, R}(t) \equiv 1 \;\text{ for }\; t \leq t_0-R^2,   \qquad  \chi_{t_0, r, R}(t) \equiv 0 \;\text{ for }\; t \geq t_0-r^2, $$$$  |\partial_t \chi_{t_0, r, R}(t) | \leq \frac{C}{R^2-r^2} \leq \frac{2C}{(R-r)^2}.
$$
Together, let us write for brevity
$$
\eta_{z_0, r, R}(x,t) = \chi_{t_0, r, R}(t) \; \varphi_{x_0, r, R}(x).
$$
Finally, for a function $f$ let us denote the oscillations at $z = (x,t)$ as follows
$$
\hat f (z)= f (z) - f_{x_0, r, R} (t), \qquad  \bar f (z) = f (z) - [f]_{x_0,R} (t),
$$
where $[f]_{x_0,R}$ is the mean value of $f$ over the ball $B(x_0,R)$.

Keeping in mind Remark \ref{ee}, it is straightforward to conclude that Lemma 2.1 of \cite{Ser08} (compare also Lemma 2.3 of of \cite{SSSZ}) holds in our case in the following form.
\begin{lemma}\label{CaccIneq}
Let $b \in {L_\infty (T_1, T_2;  BMO^{-1} (\Omega))}$ satisfy (\ref{eq:sol}). Consider any weak solution $v$ and  $q$ of (\ref{os1}) on $Q_{T_1,T_2}$. Let $Q(z_0,R) \Subset Q_{T_1,T_2}$. Then for any $t \in (t_0 - R^2, t_0)$
$$\frac{1}{2} \int\limits_\Omega |\hat v (x,t)|^2 \eta^2_{z_0, r, R} (x,t)  \,dx + \int\limits_{t_0-R^2}^{t} \int\limits_\Omega |\nabla v|^2   \eta^2_{z_0, r, R}  \,dxdt' \leq$$
\begin{equation}\label{inequality}
\leq\int\limits_{t_0-R^2}^{t}  \int\limits_\Omega \bigg( \frac{1}{2}  |\hat v|^2 (\Delta +\partial_{t'}) \eta^2_{z_0, r, R}  -  \bar d_{jl}  v_{i,l} \hat v_i  ( \eta^2_{z_0, r, R} )_{,j}  \,q \hat v  \cdot \nabla  \eta^2_{z_0, r, R}   \bigg) dx dt',
 \end{equation}

\end{lemma}

Let us assume that $Q(x_0,R_1) \Subset Q_{T_1,T_2}$ with $R<R_1$ fixed. Recall that by definition $\Gamma (z_0, R_1) = \| d \|_{L_\infty (t_0-R_1^2, t_0; \, BMO (B(x_0,R_1))}$. Identically as in \cite{Ser08} its Lemma 2.1 implies (2.7) there, we obtain from (\ref{CaccIneq}) that for any $s \in (1,6/5)$
$$\sup\limits_{t \in \, ]t_0-R^2, t_0[} \; \frac{1}{2} \int\limits_\Omega |\hat v (x,t)|^2 \eta^2_{z_0, r, R} (x,t) \, dx \, + \int\limits_{t_0-R^2}^{t_0} \int\limits_\Omega |\nabla v|^2   \eta^2_{z_0, r, R} \, dz \le $$$$ \le \frac{C}{R-r} \int\limits_{Q(z_0,R)}  |q| | \hat v| \chi^2_{t_0, r, R}\, \varphi_{x_0, r, R} \, dz \,+ $$
\begin{equation}\label{e1}
+\, C(s) \Big( \frac{\Gamma (z_0, R_1) R^{\frac{3}{s'}}}{R-r} + \frac{R^{1 + \frac{3}{s'}}}{(R-r)^2} \Big)
 \Big( \int\limits_{Q(z_0,R)} |\nabla v|^2 \, dz \Big)^\frac{1}{2} \times$$$$\times\bigg(  \int\limits_{t_0-R^2}^{t_0}  \Big( \int\limits_{B(x_0,R)} |\hat v(x,t)|^\frac{2s}{2-s} \,dx \Big)^\frac{2-s}{s} \,dt \bigg)^\frac{1}{2} \leq 
\end{equation}
$$\leq \frac{C}{R-r} \int\limits_{Q(z_0,R)}  |q| | \hat v| \chi^2_{t_0, r, R}\, \varphi_{x_0, r, R} \, dz + C(s)  \frac{(\Gamma (z_0, R_1) +1) R^{1 + \frac{3}{s'}}}{(R-r)^2}\times $$$$\times \Big( \int\limits_{Q (z_0,R)} |\nabla v|^2 \, dz \Big)^\frac{1}{2} \bigg(  \int\limits_{t_0-R^2}^{t_0}  \Big( \int\limits_{B(x_0,R)} |\hat v(x,t)|^\frac{2s}{2-s} \,dx \Big)^\frac{2-s}{s} \,dt \bigg)^\frac{1}{2}.$$

We deal with the pressure part also in a similar way as in \cite{Ser08}, pp. 332-33. Again, as in the case of (\ref{CaccIneq}), the only difference is our use of a cut-off function between any $r < R$, as opposed to a cutoff between $R$ and $2R$  in \cite{Ser08}. Nevertheless, let us present details for clarity. Since ${\rm div}\, v =0$, (\ref{os1}) implies that for any $\varphi \in C_0^\infty (\Omega)$ and a. a.  $t \in ]T_1,T_2[$
$$
\int\limits_{\Omega} q (x,t) \Delta \varphi(x) \, dx = \int\limits_{\Omega}   \bar d_{jl}  (x,t)   v_{i,l}  (x,t) \,    \varphi_{,ij}(x) \, dx.
$$
Define $q_G$ as the solution to the related very weak homogenous boundary problem in $B(x_0,R_1)$:
$$
\int\limits_{B(x_0,R_1)} q_G (x,t) \Delta \varphi(x) \, dx = \int\limits_{B(x_0,R_1)}   \bar d_{jl}  (x,t)   v_{i,l}  (x,t) \,    \varphi_{,ij} (x)\, dx \qquad $$
for all 
$\varphi \in W^2_{\frac{2s}{2-s}} (B (x_0, R_1))$ satisfying boundary condition   $\varphi(x,t)=0$ as $x\in \partial B(x_0,R_1).
$
The dual estimate implies then for a.a. $t$
\begin{equation}\label{qg}
 \Big( \int\limits_{B(x_0,R_1)} \!\!\!\!\! |q_G (x,t)|^\frac{2s}{3s-2} dx \Big)^\frac{3s-2}{2s} \!\!\! \le C(s) R^\frac{3}{s'}_1 \Gamma (z_0, R_1) \Big( \int\limits_{B(x_0,R_1)} \!\!\!\!\! |\nabla v(x,t)|^2 dx \Big)^\frac{1}{2}
\end{equation}
(compare (2.11) of  \cite{Ser08}). The remainder $q_H = q - q_G$ is harmonic on ${B (x_0,R_1)}$. Since $R<R_1$, we have then
$$
\|q_H (\cdot,t) \|_{L_\infty ({B(x_0,R)}) } \leq \frac{C}{(R_1  - R)^3} \int\limits_{B (x_0,R_1)} |q_H (x,t) | \, dx \leq$$$$\leq  \frac{C}{(R_1  - R)^3} \int\limits_{B(x_0,R_1)} (|q (x,t) |  + |q_G (x,t)|) \, dx.
$$
Use of (\ref{qg}) above implies
\begin{equation}\label{qh1}
\|q_H(\cdot,t) \|_{L_\infty ({B(x_0,R)}) }  \leq  \frac{C}{(R_1  - R)^3} \int\limits_{B(x_0,R_1)} |q (x,t) | \, dx  \, + $$$$+ \, \frac{C(s) \Gamma (z_0, R_1) R_1^\frac{3}{2}}{(R_1  - R)^3} \Big( \int\limits_{B(x_0,R_1)} |\nabla v(x,t)|^2 dx \Big)^\frac{1}{2}.
\end{equation}
We intend to use the above formulas to estimate the pressure part of (\ref{e1}). Before that, since $q=q_G + q_H$, we rewrite it as follows
\begin{equation}\label{pa}
\frac{C}{R-r} \int\limits_{Q(z_0,R)} |q| | \hat v| \chi^2_{t_0, r, R} \varphi_{x_0, r, R} \, dz \leq $$$$\leq \frac{C}{R-r}   \int\limits_{t_0-R^2}^{t_0}   \Big( \int\limits_{B(x_0,R)} |q_G(x,t)|^\frac{2s}{3s-2} dx \Big)^\frac{3s-2}{2s}  \Big( \int\limits_{B(x_0,R)} |\hat v(x,t)|^\frac{2s}{2-s}  dx\Big)^\frac{2-s}{2s} dt \,+$$$$
 +\,  \frac{C}{R-r}  \int\limits_{t_0-R^2}^{t_0} \|q_H (\cdot,t) \|_{L_\infty ({B (x_0,R)}) } \Big(  \int\limits_{B(x_0,R)} | \hat v(x,t) \eta_{z_0, r, R}(x,t)|   dx \Big) dt=$$$$  = I + II.
\end{equation}
We estimate $I$ using (\ref{qg}) 
$$
I \leq   \frac{ C(s) R^\frac{3}{s'}_1 \Gamma (z_0, R_1)}{R-r}    \Big( \int\limits_{Q (z_0,R_1)} |\nabla v |^2 dz \Big)^\frac{1}{2} \times $$$$ \times  \bigg(   \int\limits_{t_0-R_1^2}^{t_0}   \Big( \int\limits_{B(x_0,R)} |\hat v(x,t)|^\frac{2s}{2-s}  dx\Big)^\frac{2-s}{s} dt  \bigg)^\frac{1}{2} 
$$
and $II$ using  (\ref{qh1}) and next the H\"older inequality
$$
  II \leq  \frac{C}{R-r}  \int\limits_{t_0-R^2}^{t_0} \Big( \frac{1}{(R_1  - R)^3} \int\limits_{B(x_0,R_1)} |q(x,t)|  \, dx \Big) \times$$$$\times\Big(  \int\limits_{B(x_0,R)} | \hat v(x,t) \eta_{t_0, r, R}(x,t) | \, dx \Big) dt \, +
  $$
  $$
    +  \frac{C}{R-r}  \int\limits_{t_0-R^2}^{t_0}   \frac{C(s) \Gamma (z_0, R_1) R_1^\frac{3}{2}}{(R_1  - R)^3} \Big( \int\limits_{B (x_0, R_1)} |\nabla v(x,t)|^2 dx \Big)^\frac{1}{2} \times $$$$ \times \Big(  \int\limits_{B(x_0,R)} | \hat v(x,t)|  \, dx \Big) dt  \le
    $$$$
    \leq  \sup_{t \in \, ]t_0-R^2, t_0[}  \Big(  \int\limits_{B(x_0,R)} | \hat v (x,t)|^2 \eta^2_{t_0, r, R} (x,t) \, dx \Big)^\frac{1}{2}   \frac{C}{R-r}   \frac{R^\frac{3}{2}}{(R_1  - R)^3} \int\limits_{Q_{R_1} (z_0)} |q |  \, dz \,+ $$$$
    +  \frac{C}{R-r}  \frac{C(s) \Gamma (z_0, R_1) R_1^\frac{3}{2}}{(R_1  - R)^3} R^{\frac{3}{2} + \frac{3}{s'} }  \Big( \int\limits_{Q (z_0,R_1)} |\nabla v |^2 dz \Big)^\frac{1}{2}\times$$$$\times   \bigg(   \int\limits_{t_0-R_1^2}^{t_0}   \Big( \int\limits_{B(x_0,R)} |\hat v(x,t)|^\frac{2s}{2-s}  dx\bigg)^\frac{2-s}{s} dt  \bigg)^\frac{1}{2}.
$$
Finally, applying the above estimates of $I$ and $II$ to (\ref{pa}), we control the pressure term in (\ref{e1}) and arrive, after absorbing the $\sup$ term into the left-hand side, at
$$
\sup_{t \in \,]t_0-r^2, t_0[} \, \frac{1}{4} \int\limits_{B(x_0,r)} |\hat v (x,t)|^2   dx + \int\limits_{Q(z_0,r)} |\nabla v|^2  dz \leq$$$$\leq
    C(s) (\Gamma (z_0, R_1)+1)  \Big( \frac{R^{1 + \frac{3}{s'}}}{(R-r)^2} +  \frac{ R^\frac{3}{s'}_1}{R-r}  +  \frac{1}{R-r}  \frac{R_1^\frac{3}{2}}{(R_1  - R)^3} R^{\frac{3}{2} + \frac{3}{s'} } \Big)\times $$$$\times\Big( \int\limits_{Q(z_0,R_1)} |\nabla v|^2 dz \Big)^\frac{1}{2}   \bigg(   \int\limits_{t_0-R_1^2}^{t_0}   \Big( \int\limits_{B(x_0,R_1)} |\hat v(x,t)|^\frac{2s}{2-s}  dx\Big)^\frac{2-s}{s} dt  \bigg)^\frac{1}{2} +$$$$
    +\, \frac{C}{(R-r)^2}   \frac{R^3}{(R_1  - R)^6} \Big( \int\limits_{Q(z_0,R_1)} |q |  \, dz \Big)^2.
 $$
Choosing $R = \frac{R_1 + r}{2}$
we have
\begin{equation}\label{e2}
\sup_{t \in (t_0-r^2, t_0)}  \; \frac{1}{4} \int\limits_{B(x_0,r)} |\hat v (x,t)|^2   dx + \int\limits_{Q(z_0,r)} |\nabla v|^2  dz \leq $$$$
    \leq C(s)(\Gamma (z_0, R_1)+1) R_1^{\frac{3}{s'}-1} \frac{R_1^{4} } {(R_1  - r)^4}  \Big( \int\limits_{Q(z_0,R_1)} |\nabla v|^2 dz \Big)^\frac{1}{2}   \times$$$$\times\bigg(   \int\limits_{t_0-R_1^2}^{t_0}   \Big( \int\limits_{B(x_0,R_1)} |\hat v(x,t)|^\frac{2s}{2-s}  dx\Big)^\frac{2-s}{s} dt  \bigg)^\frac{1}{2} +$$$$
    +\, C  \frac{R_1^3}{(R_1  - r)^8} \Big( \int\limits_{Q(z_0,R_1)} |q | \, dz \Big)^2,
\end{equation}
valid for any $R_1 > r$. The estimate (\ref{e2}) counterparts (2.13) of  \cite{Ser08}.

We will use (\ref{e2}) twofold. Before doing so, observe that the Sobolev and H\"older inequalities yield for 
\begin{equation}\label{l}
l= \frac{6s}{12-7s} \in \; ]1,2[
\end{equation}
the inequality
\begin{equation}\label{hs}
 \int\limits_{t_0-r^2}^{t_0}   \Big( \int\limits_{B(x_0,r)} |\hat v(x,t)|^\frac{2s}{2-s}  dx\Big)^\frac{2-s}{s} dt \leq $$$$\leq C(s) \, r^\frac{2(l-1)}{l} \sup_{t \in \,]t_0-r^2, t_0[} \Big( \int\limits_{B(x_0,r)} |\hat v (x,t)|^2   dx \Big)^\frac{1}{2} \Big( \int\limits_{Q(z_0,r)} |\nabla v|^l dz \Big)^\frac{1}{l},
\end{equation}
compare estimate of $I_*$ on p.335 of  of  \cite{Ser08} ($l$ is denoted as $r$ there).

Let us return to (\ref{e2}). Firstly, using the Poincar\'e-Sobolev inequality
$$
  \Big( \int\limits_{B(x_0,R_1)} |\hat v (x,t) |^\frac{2s}{2-s}  dx\Big)^\frac{2-s}{s} \leq CR_1^\frac{6-4s}{s} \int\limits_{B(x_0,R_1)} |\nabla v(x,t)|^2 dx,
$$
we estimate only the evolutionary part of (\ref{e2}) to get
\begin{equation}
\sup_{t \in \,]t_0-r^2, t_0[}  \,\frac{1}{4} \int\limits_{B(x_0,r)} |\hat v (x,t)|^2   dx \leq $$$$\leq
    C(s) (\Gamma (z_0, R_1)+1)   \frac{R_1^4} {(R_1  - r)^4} \Big( \int\limits_{Q(z_0,R_1)} |\nabla v |^2 dz \Big) \,+$$$$+\, C  \frac{R_1^3}{(R_1  - r)^8} \Big( \int\limits_{Q(z_0,R_1)} |q |  \, dz \Big)^2. 
\end{equation}
The above estimate in the {sup} term of (\ref{hs}) yields
\begin{equation}\label{e22}
  \int\limits_{t_0-r^2}^{t_0}   \Big( \int\limits_{B(x_0,r)} |\hat v(x,t)|^\frac{2s}{2-s}  dx\Big)^\frac{2-s}{s} dt   \leq$$$$
\leq C(s) (\Gamma^\frac{1}{2} (z_0, R_1)+1)  R_1^\frac{2(l-1)}{l} \Bigg( \frac{R_1^2} {(R_1  - r)^2} \Big( \int\limits_{Q(z_0,R_1)} |\nabla v|^2 dz \Big)^\frac{1}{2} +$$$$+ C  \frac{R_1^\frac{3}{2}}{(R_1  - r)^4} \Big( \int\limits_{Q(z_0,R_1)} |q| \, dz \Big) \Bigg)  \Big( \int\limits_{Q(z_0,R_1)} |\nabla v|^l dz \Big)^\frac{1}{l}.
\end{equation}

Secondly, let us rewrite (\ref{e2}) for any $r > \rho$ in place of $R_1 > r$, dropping this time the evolutionary term
\begin{equation}
 \int\limits_{Q(z_0,\rho)} |\nabla v|^2  dz \leq $$$$
    \leq C(s) (\Gamma (z_0, r)+1) r^{\frac{3}{s'} -1}  \frac{r^4} {(r  - \rho)^4} \Big( \int\limits_{Q(z_0,r)} |\nabla v|^2 dz \Big)^\frac{1}{2}  \times $$$$\times \bigg(   \int\limits_{t_0-r^2}^{t_0}   \Big( \int\limits_{B(x_0,r)} |\hat v(x,t)|^\frac{2s}{2-s}  dx\Big)^\frac{2-s}{s} dt  \bigg)^\frac{1}{2} +$$$$+\,C\,  \frac{r^3}{(r  - \rho)^8} \Big( \int\limits_{Q(z_0,r)} |q|  \,dz \Big)^2
\end{equation}
and use for its right-hand side (\ref{e22}). Together with choosing $r = \frac{R_1 + \rho }{2}$ we arrive at

\begin{equation}\label{e3}
 \int\limits_{Q(z_0,\rho)} |\nabla v|^2  dz \leq \frac{1}{2} \int\limits_{Q(z_0,R_1)} |\nabla v|^2 dz  \, +$$$$+\, C(s) (\Gamma^5  (z_0, R_1)+1) R_1^{4 (\frac{3}{s'}-  \frac{1}{l})} \frac{R_1^{20}} {(R_1  - \rho)^{20}} \Big( \int\limits_{Q(z_0,R_1)} |\nabla v|^l dz \Big)^\frac{2}{l} $$$$
    +C  \frac{R_1^3}{(R_1 - \rho)^8} \Big( \int\limits_{Q(z_0,R_1)} |q |  \,dz \Big)^2,
\end{equation}
valid for any $R_1>\rho$ such that $Q(x_0,R_1) \Subset Q_{T_1,T_2}$.

In order to deal with the first term on the right-hand side of (\ref{e3}), let us
use the following lemma.

\begin{lemma}\label{gld}
For $0 \le t_0 < t_1$, let $h :[t_0, t_1] \to \mathbb R$ be a nonnegative bounded function. Suppose that there exists $\delta \in [0,1)$ such that for any $t_0\leq t<s\leq  t_1$ the following inequality is valid:
$$
h(t) \le \delta h(s) + \sum_{i=1}^{N} \frac{A_i (s)}{(s-t)^{\alpha_i}},
$$
 in which $\alpha_i \ge 0$, $A_i: [t_0, t_1] \to \mathbb R$ is a bounded increasing function, $i=1, \dots N$. Then, there exists a constant $C_\delta$ such that for any $t_0\leq t<s\leq t_1$
$$
h(t) \le C_\delta \sum_{i=1}^{N} \frac{A_i (s)}{(s-t)^{\alpha_i}}.$$
\end{lemma}
The proof is the same as in the classical case of constant $A_i$'s, see p. 161 of Giaquinta \cite{GiaqBook}. 

Invoking Lemma \ref{gld} with
$$
h(\rho) =  \int\limits_{Q_{\rho} (z_0)} |\nabla v|^2  dz,$$
$$
\begin{array}{ll}
A_1 (\rho) = C(s) (\Gamma^5  (z_0, \rho) +1) \rho^{20+ 4 (\frac{3}{s'}-  \frac{1}{l})} \big( \int\limits_{Q_{\rho} (z_0)} |\nabla v|^l dz \big)^\frac{2}{l}, & \alpha_1 = 20, \\
 A_2 (\rho) = \rho^3 \big( \int\limits_{Q_{\rho} (z_0)} |q |  \, dz \big)^2, & \alpha_2 = 8,
\end{array}
$$
we dispose of the first term on the right-hand side of (\ref{e3}). Consequently, choosing $R_1 = 2 \rho$ we have
$$
\frac 1{|Q(\rho)|}\int\limits_{Q(z_0,
\rho)} |\nabla v|^2  dz \leq$$$$\leq  C(s) ( \Gamma^5  (z_0, 2 \rho) +1) \rho^{4(-\frac{1}{p}  +\frac{3}{s'})} \rho^{-5} \rho^\frac{10}{p} \Big(\frac 1{|Q(2\rho)|} \int\limits_{Q(z_0,2\rho)} |\nabla v|^l dz \Big)^\frac{2}{l}+$$$$
 +\, C \Big( \frac 1{|Q(2\rho)|}\int\limits_{Q (z_0,2\rho)} |q|  \,dz \Big)^2,
 $$
which in tandem with (\ref{l}) and $s \in (1,6/5)$ implies (\ref{rhin}). Proposition  \ref{rh} is proven.

\setcounter{equation}{0}
\section{Proof of Theorem \ref{log}} 
For simplicity of the Calder\'on-Zygmund argument below, let us use in what follows both the usual (parabolic) cylinders $Q(z_0,R) = B(x_0,R) \times ]t_0 - R^2, t_0[$ and the related (parabolic) cubes $C(z_0,R) =  \{\max_{i=1,2,3} |x^i-x^i_0|<R\} \times ]t_0 - R^2, t_0[$.

Let us introduce
\begin{definition}[Local maximal function]\label{mq}
Let $G \subset \mathbb R^d$ be a fixed open set and $f \in L_1 (G)$. The {local maximal function $M_{G}$} is given by
$$
(M_{G} f)(z) = \sup\Big\{
(|f|)_C : \; {\text{cubes } \;C \;\text{ such that }\; z \in C \subset G} \Big\},
$$ where $(g)_\omega$ denotes the mean value of $g$ in $\omega$. 
\end{definition}
The following is true
\begin{lemma}\label{mf}
Let $C_0$ be a parabolic cube. Then
$$
2^{-9} \int\limits_{C_0} (M_{C_0} f) \,dz \leq \int\limits_{C_0} |f| \log \Big( e + \frac{|f|}{ (|f|)_{C_0}} \Big) dz \le 2^9 \int\limits_{C_0} (M_{C_0} f) \,dz.
$$
\end{lemma}
This result is classical in the case of the centred maximal function $M$ on $\mathbb R^d$,
 under an additional restriction that $f$ is compactly supported, see Theorem 1 of Stein \cite{Stein69}.
Lifting the compact support assumption by using the local maximal function $M_{G}$ seems virtually untapped in applications for PDEs, despite being apparently useful (in our case, trying to produce compactly supported functions, one may try to e.g. double-localise the estimates, which results in a scaling mismatch on the whole space). A range of results closely related to Lemma \ref{mf} can be found in works by Iwaniec with coauthors, e.g. \cite{Iwa4, Iwa2, Iwa1, Iwa3}. Since these papers are inspired however more by geometry-related considerations, the needed by us result seems not to be explicitly stated there. Let us therefore present the proof of Lemma \ref{mf}, emphasising that it was essentially provided to us by Piotr Haj\l asz. To this end we need the following Calder\'on-Zygmund decomposition on cubes

\begin{lemma}
\label{cz}
Let $C_0$ be a parabolic cube and $f \in L_1 (C_0)$. Fix any $t \geq (|f|)_{C_0}$. Then there exists sequence of pairwise disjoint parabolic cubes $\{C^i\}$, $C^i \subset C_0$, ${i \in \mathbb N}$ such that
\begin{subequations}
\begin{equation}\label{sta}
|f| \le t \quad \text{almost everywhere on } \;  C_0\setminus  \bigcup_{i \in \mathbb N} C^i
\end{equation}    
\begin{equation}\label{stc}
t <  (|f|)_{C^i} \le 2^8 t
\end{equation}
\end{subequations}
\end{lemma}
The only difference from the classical proof as in Stein \cite{Stein70} \S I.3.2 is a bigger constant of (\ref{stc}), related to the parabolicity of cubes.

\begin{proof}[Proof of Lemma \ref{mf}]
Let us define $E_t = \{z \in C_0: (M_{C_0} f) (z)  > t \}$. In the setting of Lemma \ref{cz}, the left inequality of  (\ref{stc}) implies $ \bigcup\limits_{i \in \mathbb N} C^i \subset E_t$. Hence
$$
\mu (E_t) \geq  \sum\limits_{i \in \mathbb N} \mu (  C^i )  \ge  2^{-8} \sum\limits_{i \in \mathbb N}  \frac{1}{t} \int\limits_{C^i} |f| \,dz = 2^{-8}   \frac{1}{t} \int\limits_{\bigcup\limits_{i \in \mathbb N} C^i} |f| \,dz,
$$
with the latter inequality given by the right inequality of (\ref{stc}). 
Since Lemma \ref{cz} implies also that $\bigcup\limits_{i \in \mathbb N} C^i  \supset \{z \in C_0:\, |f| > t \}$ up to  a zero-measure set (considering (\ref{sta})  and complements), we have in tandem with the above inequality that 
\begin{equation}\label{l3c}
\mu (E_t)  \ge  2^{-8}  \frac{1}{t} \int\limits_{\{z \in C_0:\, |f| > t \}} |f|\, dz,
\end{equation}
valid for any $t \geq (|f|)_{C_0}  =: \Lambda$. It holds
$$
2^{8} \int\limits_{C_0} M_{C_0} f dz= 2^{8} \int\limits_0^\infty \mu (E_t) \,dt \geq   \Lambda  \mu (E_{\Lambda})   +  \int\limits_{ \Lambda}^\infty  \mu (E_t) \,dt  \geq$$$$\geq \int\limits_{\{z \in C_0:\, |f| >  \Lambda \}} |f| \, dz+     \int_{\Lambda} ^\infty \frac{1}{t} \Big( \int\limits_{\{z \in C_0: |f| > t \}} |f| \, dz \Big) \,dt,
$$
see (\ref{l3c}) for the last inequality. We estimate the last integral above with help of the Tonelli theorem and  find that
$$
2^{8} \int\limits_{C_0} M_{C_0} f dz \geq \int\limits_{\{z \in C_0: \,|f| > \Lambda \}} |f| \, dz +  \int\limits_{\{z \in C_0: \,|f| > \Lambda \}}  |f| \log \frac {|f|}{\Lambda} \, dz \geq$$$$\geq   2^{-1} \int_{\{z \in C_0:\, |f| > \Lambda \}}  |f| \log \Big(e+ \frac {|f|}{\Lambda} \Big)\, dz.
$$
Since also 
$$
2^{9} \int\limits_{C_0} M_{C_0} f dz \geq 2^{9} \int\limits_{\{z \in C_0:\, |f| \le \Lambda \}}  |f| \, dz \geq $$$$\geq \int\limits_{\{z \in C_0:\, |f| \le \Lambda \}}  |f| \log (e+1) \, dz \geq \int\limits_{\{z \in C_0:\, |f| \leq \Lambda \}}  |f| \log \Big(e+ \frac {|f|}{\Lambda} \Big)\, dz,
$$
we have the right (less standard) inequality of the thesis. The remaining left inequality follows in fact from the original \cite{Stein69}. Indeed, also for the local maximal function, one has the usual weak-type estimate (i.e. a practical reverse to (\ref{l3c}))
$$
\mu (E_t)  \le  2^{8}  \frac{1}{t} \int_{\{z \in C_0:\, |f| > \frac{t}{2} \}} |f|,
$$
by the Vitali covering of $E_t$. Along the previous lines utilising (\ref{l3c}), with  inequalities  reversed, we prove now the remaining left inequality of  Lemma \ref{mf} (in fact, not needed for our further purposes).
\end{proof}

Let us return to the proof of Theorem \ref{log}. We fix a parabolic cube $C_1=C(z_1,R_1)$ such that $C_1'=C(z_1,3R_1) \Subset Q_{T_1,T_2}$. Proposition \ref{rh}, rewritten for cubes, yields for $\Gamma_1 = \Gamma (z_0, 2 R_1)$
$$
\frac 1{|C(\rho)|}\int\limits_{C(z_0,\rho)} |\nabla v|^2  dz \leq c(l) (\Gamma^5_{1} +1) \bigg(\frac 1{|C(2\sqrt 2\rho)|} \int\limits_{C(z_0,2\sqrt 2\rho)} |\nabla v |^l dz \bigg)^\frac{2}{l} +$$$$+\, c\, \bigg( \frac 1{|C(2\sqrt 2\rho)|}\int\limits_{C(z_0,2\sqrt{2} \rho)} |q |  \, dz \bigg)^2
$$
for all $C(z_0,\rho)\subset C_1$.

Since all the domains of integration of the right-hand side sit in $C(z_1,2 \sqrt{2} R_1)$, we can introduce there into integrals a smooth function $\psi$ such that $\psi \equiv 1$ on $C(z_1,2 \sqrt{2} R_1)$ and  $\psi \equiv 0$ outside $C_1' $. Hence
$$
\frac 1{|C(\rho)|}\int\limits_{C(z_0,\rho)} |\nabla v|^2  dz \leq c(l) (\Gamma^5_{1} +1) \bigg(\frac 1{|C(2\sqrt 2\rho)|} \int\limits_{C(z_0,2\sqrt 2\rho)} |\nabla v |^l\psi \,dz \bigg)^\frac{2}{l} +$$$$+\, c\, \bigg( \frac 1{|C(2\sqrt 2\rho)|}\int\limits_{C(z_0,2\sqrt{2} \rho)} |q | \psi \, dz \bigg)^2
$$
for all $C(z_0,\rho)\subset C_1$.

Recalling Definiton \ref{mq} we have then
$$
M_{C_{1}} (|\nabla v|^2) (z)  \leq c(l) (\Gamma^5_{1} +1)  M^\frac{2}{l}_{C(z_1,2 \sqrt{2} R_1)} (|\nabla v|^l \psi) (z) +c M^2_{C(z_1,2  \sqrt{2} R_1)} (|q|  \psi) (z) $$$$
 \leq c(l) (\Gamma^5_1 +1)  M^\frac{2}{l}_{\mathbb R^4} (|\nabla v|^l \psi) (z) + c M^2_{\mathbb R^4} (|q|  \psi) (z)
$$
and consequently
$$
\int\limits_{C_{1}} M_{C_{1}} (|\nabla v|^2) \, dz \leq c(l) (\Gamma^5_{1} +1) \int\limits_{\mathbb R^4}  M^\frac{2}{l}_{\mathbb R^4} (|\nabla v|^l \psi) \, dz +\, c\int\limits_{\mathbb R^4}   M^2_{\mathbb R^4} (|q|  \psi)  \, dz.
$$
Observe that $M_{\mathbb R^4}$ is the usual non-centred maximal function with respect to parabolic cubes. Since it enjoys the strong $L_p$-property, compare \cite{Stein93} \S I.3.1, the above inequality implies
$$
\int\limits_{C_{1}} M_{C_{1}} (|\nabla v|^2) \, dz \leq c(l) (\Gamma^5_{1} +1) \int\limits_{\mathbb R^4}  |\nabla v|^2 \psi^\frac{2}{l}  \, dz + c \int\limits_{\mathbb R^4} |q|^2  \psi^2 \, dz  \leq$$$$\leq c(l) (\Gamma^5_{1} +1) \int\limits_{C_1'}  |\nabla v|^2 \, dz + c\int\limits_{C_1'}  |q|^2 \, dz< \infty,
$$
hence  Lemma \ref{mf} yields
$$
\int\limits_{C_{1}} |\nabla v| \log \Big( e + \frac{|\nabla v|}{ (|\nabla v|)_{C_1}} \Big) dz \leq 2^9 c(l) (\Gamma^5_{1} +1) \int\limits_{C_{1}'}  |\nabla v|^2 \, dz+ 2^9 c \int\limits_{C_{1}'} |q|^2 \, dz.
$$
Returning to parabolic cylinders gives Theorem \ref{log}.

\vskip 5mm
\noindent
{\bf Acknowledgements}
The authors are indebted to Piotr Haj\l asz in relation to local maximal functions.

\noindent
J. Burczak was supported by MNiSW "Mobilno\'s\'c Plus" grant \\1289/MOB/IV/2015/0. \\
G. Seregin was supported by the grant RFBR 17-01-00099-a.

\setcounter{equation}{0}
\section{Appendix I} 
Here, we are going to prove (\ref{MazVer}).
Indeed, we have 
$$(D\nabla u):\nabla v=u_{i,l}d_{jl}v_{i,j}=u_{i,l}\epsilon_{jls}v_{i,j}\omega_s.$$
Since $\omega$ is an $BMO$ function, it suffices to show that for any $s=1,2,3$, the function 
$$x\mapsto u_{i,l}(x)\epsilon_{jls}v_{i,j}(x)
$$
belongs to the Hardy space and to  find the corresponding estimates, compare e.g. \S VII.3 of \cite{Stein70} about duality between Hardy and  $BMO$ spaces. To this end, let us fix a standard mollifier $\Phi_\varrho$
and consider the function 
$$H_s(x):=\sup\limits_{\varrho>0}|(\Phi_\varrho\star (u_{i,l}\epsilon_{jls}v_{i,j}))(x)|.$$
Taking into account properties of the Levi-Civita tensor, we have 
$$H_s(x)=\sup\limits_{\varrho>0}|(\Phi_\varrho\star (\overline u_{i}\epsilon_{jls}v_{i,j})_{,l})(x)|,$$
where $\overline u=u-[u]_{B(x,\varrho)}$. After integration by parts and applying the estimate $|\nabla \Phi_\varrho|\leq c\varrho^{-4}$, we find
$$H_s(x)\leq \sup\limits_{\varrho>0}
\frac c\varrho
\frac 1{|B(\varrho)|}
\int\limits_{B(x,\varrho)}|\overline u||\nabla v|dy
\leq \frac c\varrho\bigg(\frac 1{|B(\varrho)|}
\int\limits_{B(x,\varrho)}|\overline u|^3dy\bigg)^\frac 13\times$$$$\times \bigg(\frac 1{|B(\varrho)|}
\int\limits_{B(x,\varrho)}|\nabla v|^\frac 32dy\bigg)^\frac 23.$$
Now, we can use Poincar\'e-Sobolev inequality and pass to the standard centred (Hardy-Littlewood) maximal functions, denoted by $M$, thus obtaining 
$$H_s(x)\leq c\sup\limits_{\varrho>0}\bigg(\frac 1{|B(\varrho)|}
\int\limits_{B(x,\varrho)}|\nabla u|^\frac 32dy\bigg)^\frac 23 \bigg(\frac 1{|B(\varrho)|}
\int\limits_{B(x,\varrho)}|\nabla v|^\frac 32dy\bigg)^\frac 23\leq $$
$$\leq cM^\frac 23 ({|\nabla u|^\frac 32}) (x)M^\frac 23({|\nabla v|^\frac 32})(x).
$$
Integration over $\mathbb R^3$, together with $L_p$-estimates   for maximal functions gives us
$$\|H_s\|_1\leq c\Big(\int\limits_{\mathbb R^3}M^\frac 43 ({|\nabla u|^\frac 32})(x)dx\Big)^\frac 12
\Big(\int\limits_{\mathbb R^3}M^\frac 43({|\nabla v|^\frac 32})(x)dx\Big)^\frac 12\leq c\|\nabla u\|_2\|\nabla v\|_2$$
for any $s=1,2,3$. Therefore, by definition, for any $s=1,2,3$, $H_s$ belongs to the Hardy space. Now, estimate (\ref{MazVer}) follows from duality between Hardy and $BMO$ spaces.


\setcounter{equation}{0}
\section{Appendix II}
Here we state an existence theorem for the Cauchy problem for system (\ref{os1}), compare Remark \ref{1strem}. To this end we need to introduce certain energy spaces. 
First, we 
let 
$$C^\infty_{0,0}(\Omega)=\{v\in C^\infty_0(\Omega):\,\,{\rm div} v=0\}$$ and then

$${\stackrel{\circ} J}_p(\Omega)=[C^\infty_{0,0}(\Omega)]^{L_p(\Omega)},$$
${\stackrel{\circ} J}{_p^1}(\Omega)$ is the closure of the set $C^\infty_{0,0}(\Omega)$
with respect to the semi-norm
$$|v|_{p,1,\Omega}=\Big(\int\limits_{\Omega}|\nabla v|^pdx\Big)^\frac 1p.$$
If $\Omega=\mathbb R^3$, we shall drop $\Omega$ in the notation of the spaces. We denote $\mathbb R^3 \times \mathbb R_+$ briefly by $Q_+$.
\begin{theorem}
	\label{Cauchy} Given a skew symmetric tensor $b\in L_\infty(BMO)$ and initial velocity $ u_0\in {\stackrel{\circ} J}_2$, there exists a unique pair $v$ and $q$ satisfying the following properties:
	\begin{itemize}
\item[(i)] 	$v\in L_\infty(0,\infty;{\stackrel{\circ} J}_2)\cap L_2(0,\infty;{\stackrel{\circ} J}{_2^1}), \qquad q\in L_2(Q_+);$

\item[(ii)] 	$v$ and $q$ satisfy the problem (\ref{os1}) in the sense of distributions;
	
\item[(iii)] 
	the function
	$$t\mapsto \int\limits_{\mathbb R^3}v(x,t)\cdot w(x)dx
	$$ is continuous at any $t\geq0$ for each $w\in L_2(\mathbb  R^3)$;
\item[(iv)] 	$\|v(\cdot,t)-u_0(\cdot)\|_2\to 0$ as $t\to0$;
	
\item[(v)] 
	for all $t\geq0$
	$$\frac 12 \int\limits_{\mathbb R^3}|v(x,t)|^2dx+\int\limits^t_0\int\limits_{\mathbb R^3}|\nabla v|^2dxdt'\leq \frac 12 \int\limits_{\mathbb R^3}|u_0|^2dx;$$
	
\item[(vi)] 
	for all $t\geq0$
	$$\int\limits_{\mathbb R^3}\varphi|v(x,t)|^2dx+2\int\limits^t_0\int\limits_{\mathbb R^3}\varphi|\nabla v|^2dxdt'= $$$$= \int\limits^t_0\int\limits_{\mathbb R^3}
	(|v|^2(\partial_t+\Delta) \varphi-2D\nabla v:v\otimes \nabla \varphi+2qv\cdot\nabla\varphi)dxdt' $$
	for any non-negative $\varphi\in C^\infty_0(Q_+)$.
	\end{itemize}
	\end{theorem}
The proof of the theorem relies essentially  on the estimate (\ref{MazVer}). Observe that it is also applicable to the pressure equation
$$-\Delta q={\rm div}\,(D\nabla v)
$$	
hence gives the estimate  for the pressure
$$\|q\|_{2,Q_+}\leq c \|d\|_{L_\infty(BMO)}\|\nabla v\|_{2,Q_+}.
$$
Further details are standard.

\end{document}